\documentclass[12pt]{amsart}
\usepackage{mathrsfs}
\usepackage{amsfonts}
\usepackage{txfonts}
\usepackage[arrow,matrix]{xy}
\usepackage{amsmath,amssymb,amscd,bbm,amsthm,mathrsfs,dsfont,color,hyperref}
\usepackage{changes}

\theoremstyle{plain} \textwidth=36pc \textheight=51pc

\newtheorem{theorem}{Theorem}[section]
\newtheorem{lemma}[theorem]{Lemma}
\newtheorem{exa}[theorem]{Example}
\newtheorem{thm}[theorem]{Theorem}
\newtheorem{prop}[theorem]{Proposition}

\theoremstyle{definition}
\newtheorem{defn}[theorem]{Definition}
\newtheorem{remark}[theorem]{Remark}
\numberwithin{equation}{section}

\newcommand{\Z}{\mathbb{Z}}

\newcommand{\Hom}{\mathop{\mathrm{Hom}}\nolimits}
\newcommand{\HG}{\mathop{\mathrm{H}}\nolimits}

\newcommand{\field}{\mathop\mathrm{k}}
\newcommand{\DGM}{\mathop\mathrm{DGrMod}}
\newcommand{\DGL}{\mathop\mathrm{DGrLie}}
\newcommand{\DGP}{\mathop\mathrm{DGrP}}
\newcommand{\DGA}{\mathop\mathrm{DGA}}

\newcommand{\DGPA}{\mathop\mathrm{DGPA}}

\begin{document}
\title[DG Poisson algebra and its universal enveloping algebra]{DG Poisson algebra and its universal enveloping algebra}

\author{Jiafeng L\"u}
\address{L\"u: Department of Mathematics, Zhejiang Normal University, Jinhua, Zhejiang 321004, P.R. China}
\email{jiafenglv@zjnu.edu.cn, jiafenglv@gmail.com}

\author{Xingting Wang}
\address{Wang: Department of Mathematics, Temple University, Philadelphia 19122, USA }
\email{xingting@temple.edu}

\author{Guangbin Zhuang}
\address{Zhuang: Department of Mathematics,
University of Southern California, Los Angeles 90089-2532, USA}
\email{gzhuang@usc.edu}

%\thanks{*corresponding author}

\begin{abstract}
In this paper, we introduce the notions of differential graded (DG) Poisson algebra and DG Poisson module. Let $A$ be any DG Poisson algebra. We construct the universal enveloping algebra of $A$ explicitly, which is denoted by $A^{ue}$. We show that $A^{ue}$ has a natural DG algebra structure and it satisfies certain universal property. As a consequence of the universal property, it is proved that the category of DG Poisson modules over $A$ is isomorphic to the category of DG modules over $A^{ue}$. Furthermore, we prove that the notion of universal enveloping algebra $A^{ue}$ is well-behaved under opposite algebra and tensor product of DG Poisson algebras. Practical examples of DG Poisson algebras are given throughout the paper including those arising from differential geometry and homological algebra. 

\end{abstract}
\maketitle

\section*{Introduction}
The Poisson bracket was originally introduced by French Mathematician Sim{\'e}on Denis Poisson in search for integrals of motion in Hamiltonian mechanics. For Poisson algebras, many important generalizations have been obtained in both commutative and noncommutative settings recently: Poisson orders \cite{BG}, graded Poisson algebras \cite{C, CFL}, noncommutative Leibniz-Poisson algebras \cite{CD}, left-right noncommutative Poisson algebras \cite{CDL}, Poisson PI algebras \cite{MPR}, double Poisson algebras \cite{Vd}, noncommutative Poisson algebras \cite{Xup}, Novikov-Poisson algebras \cite{X} and Quiver Poisson algebras \cite{YYZ}. One of the interesting aspects of Poisson algebras is the notion of Poisson universal enveloping algebra, which was first introduced by Oh in order to describe the category of Poisson modules \cite{Oh1}. Since then, many progresses have been made in various directions of studying Poisson universal enveloping algebras, such as universal derivations and automorphism groups \cite{U},  Poisson-Ore extensions \cite{LWZ1}, Poisson Hopf algebras \cite{LWZ2}, and deformation theory \cite{YYY}.

In this paper, our aim is to study Poisson algebras and their universal enveloping algebras in the differential graded setting. Roughly speaking, a DG Poisson algebra is a differential graded commutative algebra together with a Poisson bracket satisfying certain compatible conditions. The Poisson bracket in our definition is assumed to have arbitrary degree rather than zero, which seems to be more natural as it appears in various examples like Schouten-Nijenhuis bracket in differential geometry and Gerstenhaber-Lie bracket in homological algebra. There are many equivalent ways to define the universal enveloping algebra of a DG Poisson algebra, among which we choose to use the explicit generators and relations; see Definition \ref{Def:DGPA} and Remark \ref{Rmk}. Regarding the universal property, we succeed to generalize it to the DG setting in Proposition \ref{Prop:Universal}, which induces an equivalence of two module categories as Theorem \ref{THM:E}. Moreover, by considering the opposite algebra and the tensor product of DG Poisson algebras, we can view the universal enveloping algebra as a tensor functor from the category of DG Poisson algebras to the category of DG-algebras, which is summarized in Theorem \ref{THM:T}. Moreover, any commutative Poisson algebra gives rise to a Lie-Rinehart pair associated to its K{\"a}hler differential \cite[\S 5.2]{LWZ1}, which might provide an alternative approach to the results in our paper once we add the DG structure.  

Differential graded algebras arise naturally in algebra, representation theory and algebraic topology. For instance, Koszul complexes, endomorphism algebras, fibers of ring homomorphisms, singular chain and cochain algebras of topological spaces, and bar resolutions all admit natural differential graded algebra structures. Moreover, differential graded commutative algebra provides powerful techniques for proving theorems about modules over commutative rings \cite{BS}. On the other side, using deformation quantization theory \cite{EK, H, Kon}, Poisson brackets can be used as a tool to study noncommutative algebras. For example, by considering the Poisson structure on the center of Weyl algebras, Belov and Kontsevich proved that the Jacobian Conjecture is stably equivalent to the Dixmier Conjecture \cite{BK}. Therefore, as a generalization of both structures, we expect more study of DG Poisson algebras for its own interests, and further applications in a wide range of mathematical fields.

The paper is organized as follows. Some background materials are reviewed in Section \ref{Section:BG}. The definition of DG Poisson algebra is given in Section \ref{Section:Def}, and the universal enveloping algebra of a DG Poisson algebra is defined in Section \ref{Section:Universal}. Basic properties of universal enveloping algebras are discussed in Section \ref{Section:Prop}. Several further examples of DG Poisson algebras are provided in Section \ref{Section:Example}. In last Section \ref{Section:Dis}, future projects related to this paper are discussed. 

\subsection*{Acknowledgments} The authors first want to give their sincere gratitudes to James Zhang for introducing them this project. They also want to thank Yanhong Bao, Jiwei He, Xuefeng Mao, Cris Negron and James Zhang for many valuable discussions and suggestions on this paper. Additional thanks are given to the anonymous referees for several suggestions that improved the exposition of the paper. The first author is partially supported by National Natural Science Foundation of China [11001245, 11271335 and 11101288].

\medskip
\section{Preliminaries}\label{Section:BG}
Throughout we work over a base field $\field$, all vector spaces and linear maps are over $\field$. In what follows, an unadorned $\otimes$ means $\otimes_{\field}$ and $\Hom$ means $\Hom_{\field}$. All gradings are referred to $\Z$-gradings with index occurs in upper superscript regarding cohomology. All elements and homomorphisms will be homogeneous in the graded setting. We refer the reader to \cite{Rational, Quillen} for the following well-known definitions. 

A graded (associative) algebra means a $\Z$-graded algebra $A=\bigoplus_{i\in \Z}A^i$ with $\field\subseteq A^0$ and $A^iA^j\subseteq A^{i+j}$ for all $i,j\in \Z$. Moreover, $A$ is said to be graded commutative if $ab=(-1)^{|a||b|}ba$ for all elements $a,b\in A$, where $|a|,|b|$ denote the degree of $a,b$ in $A$. 

A differential graded algebra (or DG-algebra, for short) is a graded algebra $A$ together with a map $d_A: A\to A$ of degree 1 such that \begin{enumerate} \item (differential) $d_A^2=0$; and \item (graded Leibniz rule) $d_A(ab)=d_A(a)b+(-1)^{|a|}ad_A(b)$ for all elements $a,b\in A$. \end{enumerate} Note that the cohomology of a DG-algebra is always a graded algebra.  

Let $A$ be a DG-algebra. A left differential graded (DG) module over $A$ is a graded left $A$-module $M=\bigoplus_{i\in \Z}M^i$ together with a linear map $d_M: M\to M$ of degree 1 such that \begin{enumerate} \item $d_M^2=0$; and \item $d_M(a\cdot m)=d_A(a)\cdot m+(-1)^{|a|}a\cdot d_M(m)$ for all elements $a\in A$ and $m\in M$, where $\cdot$ denotes the $A$-module action on $M$. \end{enumerate} Right DG $A$-modules and DG $A$-bimodules can be defined in the similar manner. The category of left DG $A$-modules (resp. right DG $A$-modules) will be denoted by $\DGM$-$A$ (resp. $A$-$\DGM$). For any $M,N\in \DGM$-$A$, let $\Hom_A(M,N)$ be the set of all $A$-module morphisms $f: M\to N$ of degree 0 such that $fd_M=d_Nf$. 

A differential graded (DG) Lie algebra is a graded vector space $L=\bigoplus_{i\in \Z} L^i$ together with a bilinear Lie bracket $[-,-]: L\otimes L\to L$ of degree $p$ and a differential $d: L\to L$ of degree 1 ($d^2=0$) such that \begin{enumerate} \item (antisymmetry) $[a,b]=-(-1)^{(|a|+p)(|b|+p)}[b,a]$; \item (Jacobi identity) $[a,[b,c]]=[[a,b],c]+(-1)^{(|a|+p)(|b|+p)}[b,[a,c]]$; and \item $d[a,b]=[d(a),b]+(-1)^{(|a|+p)}[a,d(b)]$ for all elements $a,b,c\in L$. \end{enumerate} It is worth pointing out that in above definition, we assume that the Lie bracket has arbitrary degree $p$ rather than $p=0$ as usual. It applies to many cases, for example, the Gerstenhaber Lie bracket defined for the Hochschild cohomology for any associative algebra has degree $-1$ ($p=-1$ and $d=0$). Note that Jacobi identity of the Lie bracket may also be expressed in a more symmetrical form: for any $a,b,c\in L$, we have
$$
 (-1)^{(|a|+p)(|c|+p)}[a,[b,c]]+ (-1)^{(|b|+p)(|a|+p)}[b,[c,a]]+(-1)^{(|c|+p)(|b|+p)}[c,[a,b]]=0.
 $$

Let $L$ be a DG Lie algebra with Lie bracket of degree $p$. A left differential graded (DG) Lie module over $L$ is a graded vector space $M=\bigoplus_{i\in \Z}M^i$ together with a differential $d_M: M\to M$ of degree 1 ($d_M^2=0$) and a bilinear map $[-,-]_M: L\otimes M \to M$ of degree $p$ such that \begin{enumerate} \item $[[a,b]_L,m]_M=[a,[b,m]_M]_M+(-1)^{(|a|+p)(|b|+p)}[b,[a,m]_M]_M$; and \item $d_M([a, m]_M)=[d_L(a), m]_M+(-1)^{(|a|+p)}[a,d_M(m)]_M$ for all elements $a\in L$ and $m\in M$. \end{enumerate} Right DG Lie $L$-modules and DG Lie $L$-bimodules can be defined in the similar manner. The category of left DG Lie $L$-modules (resp. right DG Lie $L$-modules) will be denoted by $\DGL$-$L$ (resp. $L$-$\DGL$). For any $M,N\in \DGL$-$A$, let $\Hom_L(M,N)$ be the set of all Lie $L$-module morphisms $f: M\to N$ of degree 0 such that $fd_M=d_Nf$.

\section{DG Poisson algebras and modules}\label{Section:Def}
In literature, a Poisson algebra usually means a commutative algebra together with a Lie bracket which is a biderivation \cite[Definition 2.1]{LWZ1}. The notion of DG Poisson algebra can be considered as a combination of commutative algebras and Lie algebras in the differential graded setting. In the following definition, we allow the Poisson bracket to have arbitrary degree $p$.

\begin{defn}\label{D:DGPA}
A differential graded Poisson algebra (or DG Poisson algebra, for short) is a quadruple $(A,\cdot,\{-.-\}, d)$, where $A=\bigoplus_{i\in \Z} A^i$ is a graded vector space such that  
\begin{itemize}
\item[(i)] $(A,\cdot,d)$ is a graded commutative DG-algebra with differential $d: A\to A$ of degree 1.
\item[(ii)] $(A,\{-,-\},d)$ is a DG Lie algebra with Poisson bracket $\{-,-\}: A\otimes A\to A$ of degree $p$.
\item[(iii)] (Poisson identity): $\{a,b\cdot c\}=\{a,b\}\cdot c+(-1)^{(|a|+p)|b|}b\cdot \{a,c\}$ for all elements $a,b,c\in A$.
\end{itemize}
\end{defn}

\begin{remark}
A noncommutative version of a DG Poisson algebra can be made without ``graded commutative'' in Definition \ref{D:DGPA} (i).
\end{remark}

In the following, we will consider Poisson modules over a DG Poisson algebra $A$ with Poisson bracket of degree $p$. In order to make compatible with the degree of the Poisson bracket of $A$, we additionally require that the bracket $A\otimes M\to M$ in the Poisson module structure for any Poisson module $M$ shares the same degree $p$ with the degree of the Poisson bracket of $A$.  

\begin{defn}\label{Def:DGP}
A left differential graded (DG) Poisson module over $A$ is a quadruple $(M,*,\{-,-\}_M,d_M)$, where $M=\bigoplus_{i\in \Z} M^i$ is a graded vector space with differential $d_M:M\to M$ of degree 1 such that
\begin{itemize}
\item[(i)] $(M,*,d_M)\in \DGM-A$ via $A$-module action $*: A\otimes M\to M$ of degree 0.  
\item[(ii)] $(M,\{-.-\}_M,d_M)\in \DGL-A$ via Lie $A$-module action $\{-.-\}_M: A\otimes M\to M$ of degree $p$.
\item[(iii)] $\{a,b*m\}_M=\{a,b\}_A*m+(-1)^{(|a|+p)|b|}b*\{a,m\}_M$, and 
\item[(iv)] $\{a\cdot b,m\}_M=a*\{b,m\}_M+(-1)^{|a||b|}b*\{a,m\}_M$ for all elements $a,b\in A$ and $m\in M$.
\end{itemize}
Right DG Poisson $A$-modules and DG Poisson $A$-bimodules are defined in the similar manner. The category of left DG Poisson $A$-modules (resp. right DG Poisson $A$-modules) will be denoted by $\DGP$-$A$ (resp. $A$-$\DGP$). For any $M,N\in \DGP$-$A$, let $\Hom_A(M,N)$ be the set of all morphisms $f: M\to N$ of degree 0, which are both $A$-module and Lie $A$-module morphisms satisfying $fd_M=d_Nf$. 
\end{defn}

\begin{exa}\label{eg1}
{\rm The following are some trivial examples of DG Poisson algebras.
\begin{enumerate}
  \item[(i)] Let $(A,\cdot, \{-,-\},d)$ be a DG Poisson algebra with Poisson bracket of degree $0$. Then the degree 0 piece $(A^0,\cdot, \{-,-\})$ is an ordinary Poisson algebra. Conversely, any Poisson algebra can be considered as a DG Poisson algebra concentrated in degree 0 with trivial differential.
  \item[(ii)] Graded commutative DG-algebras are DG Poisson algebras with trivial Poisson bracket.
  \item[(iii)] Graded Poisson algebras are DG Poisson algebras with trivial differential.
  \item[(iv)] DG Lie algebras are DG Poisson algebras with trivial product added by a copy of the base field $\field$.
\end{enumerate}}
\end{exa}

\begin{exa}\label{EX:DGV}
{\rm 
Let $(V, d_V)$ be a differential graded vector space with differential $d_V: V\to V$ of degree 1. There are two natural ways to construct DG Poisson algebras from $V$. One is to take the $\Hom$ complex of $V$, denoted by $\Hom(V, V)$. It is easy to check that $\Hom(V, V)$ is a noncommutative DG-algebra with the product given by the composition, and the differential given by $d(f)=d_Vf-(-1)^{|f|}fd_V$ for any $f\in \Hom(V,V)$. Moreover, one sees that $\Hom(V,V)$ becomes a DG Poisson algebra, where the Poisson bracket is given by the graded commutator defined by $[f,g]:=fg-(-1)^{|f||g|}gf$ for any $f,g\in \Hom(V,V)$.

The other way is to consider the graded symmetric algebra $S(V)$ over $V$ such that  
\[
S(V):=T(V)/(f\otimes g-(-1)^{|f||g|}g\otimes f),
\]
where $T(V)$ is the tensor algebra over $V$, and $f,g\in V$. Therefore, $S(V)$ becomes a DG Poisson algebra via (1) (differential) $d(f)=d_Vf-(-1)^{|f|} fd_V$, and (2) (Poisson bracket) $\{f,g\}:=[f,g]$ for any $f,g\in V$.
}
\end{exa}
\section{Universal enveloping algebras of DG Poisson algebras}\label{Section:Universal}
The universal enveloping algebra of an ordinary Poisson algebra is given in \cite{Oh1}. Our aim is to generalize the definition to the differential graded setting. Let $A$ be a DG Poisson algebra with Poisson bracket $\{-,-\}$ of degree $p$.   
\begin{defn}\label{Def:DGPA}
Let $M_A =\{M_a: a\in A\}$ and $H_A =\{H_a: a\in A\}$ be two copies of the graded vector space $A$ endowed with two linear isomorphisms $M: A\to M_A$ sending $a$ to $M_a$ and $H: A\to H_A$ sending $a$ to $H_a$. The universal enveloping algebra $A^{ue}$ of $A$ is defined to be the quotient algebra of the free algebra generated by $M_A$ and $H_A$, subject to the following relations:
\begin{itemize}
\item[(i)] $M_{ab}=M_aM_b$,
\item[(ii)] $H_{\{a,b\}}=H_aH_b-(-1)^{(|a|+p)(|b|+p)}H_bH_a$,
\item[(iii)] $H_{ab}=M_aH_b+(-)^{|a||b|}M_bH_a$,
\item[(iv)] $M_{\{a,b\}}=H_aM_b-(-1)^{(|a|+p)|b|}M_bH_a$,
\item[(v)] $M_1=1$
\end{itemize}
for all elements $a,b\in A$. 
\end{defn}
\begin{lemma}
The universal enveloping algebra $A^{ue}$ has a natural DG-algebra structure induced by $A$ such that $|M_a|=|a|$, $|H_a|=|a|+p$ and $d(M_a)=M_{d(a)}$, $d(H_a)=H_{d(a)}$ for any $a\in A$. 
\end{lemma}
\begin{proof}
It is clear that all the relations (i)-(v) in Definition \ref{Def:DGPA} are homogeneous. Then, $A^{ue}$ is a graded algebra. It remains to show that $d$ preserves all the relations, which is routine to check, e.g., for (ii) we have
\begin{align*}
&d(H_{\{a,b\}})\\
=&\ H_{d\{a,b\}}\\
=&\ H_{\{d(a),b\}}+(-1)^{(|a|+p)}H_{\{a,d(b)\}}\\
=&\ H_{d(a)}H_b-(-1)^{(|d(a)|+p)(|b|+p)}H_bH_{d(a)}+(-1)^{(|a|+p)}(H_aH_{d(b)}-(-1)^{(|a|+p)(|d(b)|+p)})H_{d(b)}H_a)\\
=&\ (H_{d(a)}H_b+(-1)^{(|a|+p)}H_aH_{d(b)})-(-1)^{(|a|+p)(|b|+p)}(H_{d(b)}H_a+(-1)^{(|b|+p)}H_bH_{d(a)})\\
=&\ (d(H_a)H_b+(-1)^{|H_a|}H_ad(H_b))-(-1)^{(|a|+p)(|b|+p)}(d(H_b)H_a+(-1)^{|H_b|}H_bd(H_a))\\
=&\ d(H_aH_b-(-1)^{(|a|+p)(|b|+p)}H_bH_a).
\end{align*}
This completes the proof. 
\end{proof}

The universal enveloping algebra of an ordinary Poisson algebra can be described by certain universal property; see \cite[\S1.2]{LWZ2}. The same thing happens to any DG Poisson algebra $A$, which is determined by the similar universal property. We say a triple $(B,f,g)$ has property $\mathscr P$ with respect to $A$ (or $\mathscr P$-triple, for short) if
\begin{itemize}
\item[(P1)] $B$ is a DG-algebra and $f:A\to B$ is a DG-algebra map of degree 0;
\item[(P2)] $g:(A,\{-,-\})\to B$ is a DG Lie algebra map of degree $p$, where the Lie bracket on $B$ is given by the graded commutator;
\item[(P3)] $f({\{a,b\}})=g(a)f(b)-(-1)^{(|a|+p|)b|}f(b)g(a)$, and
\item[(P4)] $g(ab)=f(a)g(b)+(-1)^{|a||b|}f(b)g(a)$, for all $a,b\in A$.
\end{itemize}  

Note that in Definition \ref{Def:DGPA}, $M_A, H_A$ induces two linear maps from $A$ to $A^{ue}$, where we abuse to keep the same notations as $M, H: A\to A^{ue}$. It is clear that the triple $(A^{ue},M,H)$ has property $\mathscr P$, which is universal in the following sense.
\begin{prop}\label{Prop:Universal}
For any $\mathscr P$-triple $(B,f,g)$, there exists a unique DG-algebra map $\phi: A^{ue}\to B$ of degree 0 such that the following diagram
\[
\xymatrix{
A\ar[rr]^-{M}_-{H}\ar[dr]^-{f}_-{g} && A^{ue}\ar@{-->}[dl]^-{\exists ! \phi}\\
& B&
}
\]
bi-commutes, i.e., $f=\phi M$ and $g=\phi H$.
\end{prop}
\begin{proof}
Suppose the map $\phi$ exists. We get $\phi(M_a)=f(a)$ and $\phi(H_a)=g(a)$ for all $a\in A$ by using bi-commutativity $f=\phi M$ and $g=\phi H$. In order to show that $\phi$ is well-defined on $A^{ue}$, we need to show that $\phi$ preserves all the relations (i)-(v) in Definition \ref{Def:DGPA}. (P1) implies that $\phi(M_{ab})=f(ab)=f(a)f(b)=\phi(M_a)\phi(M_b)$ for any $a,b\in A$ and $\phi(M_1)=f(1)=1=\phi(1)$. So $\phi$ preserves (i) and (v). Similarly, $\phi$ preserves (ii) because of (P2); (iii) is preserved by $\phi$ by (P4); and (P3) implies that (iv) is preserved by $\phi$. Moreover, it is clear that $\phi$ commutes with the differentials of $A^{ue}$ and $B$ for $f$ and $g$ are DG maps. Finally, the uniqueness of $\phi$ is obvious. This completes the proof. 
\end{proof}
 
\begin{remark}\label{Rmk}
The other way to define the universal enveloping algebra of an ordinary Poisson algebra is though smash product \cite[Definition 4.1.3]{Mo} modulo certain relations; see \cite[\S 2]{YYZ}. We can carry the same idea for any DG Poisson algebra $A$. Consider $(A,\{-,-\})$ as a DG Lie algebra, and denote by $\mathcal U(A)$ its universal enveloping algebra. It is important to point out that $\mathcal U(A)$ is a differential graded Hopf algebra with respect to the total grading coming from the grading of $A$, where elements in $A\subset \mathcal U(A)$ are all primitive. It is routine to check that $A$ becomes a differential graded $\mathcal U(A)$-module algebra via $h\cdot a:=\{h,a\}$ for all elements $h,a\in A$. As a consequence, the universal enveloping algebra $A^{ue}$ is the quotient algebra of the smash product $A\#\mathcal U(A)$, subject to the relations $1\#(ab)-a\#b-(-1)^{|a||b|}b\#a$ for all elements $a,b\in A$. 
\end{remark}

\section{Some basic properties}\label{Section:Prop}
Recall in Example \ref{EX:DGV}, let $V$ be a graded vector space with differential $d_V$ of degree 1. Denote by $\Hom(V,V)$ all the homogenous maps from $V$ to itself. which has a natural differential given by $d(f)=d_V f-(-1)^{|f|}f d_V$ for any $f\in \Hom(V,V)$. Moreover, one sees that $\Hom(V,V)$ becomes a DG-algebra via composition of maps such that $d(fg)=d(f)g+(-1)^{|f|}fd(g)$ for any $f,g\in \Hom(V,V)$. 

Now, let $A$ be a DG Poisson algebra with Poisson bracket $\{-,-\}_A$ of degree $p$. By Definition \ref{Def:DGP}, any DG Poisson $A$-module structure on $V$ requires two linear maps from $A$ to $\Hom(V,V)$, i.e., one $A$-module action $*: A\otimes V\to V$ of degree 0 and another Lie $A$-module action $\{-.-\}_V: A\otimes V\to V$ of degree $p$ satisfying conditions described in Definition \ref{Def:DGP} (i)-(iv). The following lemma is straightforward.

\begin{lemma}\label{Lemma:P}
The DG vector space $(V,d_V)$ belongs to $\DGP$-$A$ if and only if there exits a $\mathscr P$-triple $(\Hom(V,V), f, g)$ such that the $A$-module and Lie $A$-module actions on $V$ are given by $f$ and $g$ correspondingly.
\end{lemma}

We will see that the notion of universal enveloping algebra for DG Poisson algebra is proper in its usual sense as there exists an equivalence of two module categories.

\begin{thm}\label{THM:E}
The category of left (resp. right) DG Poisson modules over $A$ is equivalent to the category of left (resp. right) DG modules over $A^{ue}$, i.e.,
$$
\DGP-A\equiv \DGM-A^{ue}\ (\mbox{resp.}\ A-\DGP\equiv A^{ue}-\DGM).
$$
\end{thm}
\begin{proof}
Let $M\in \DGP$-$A$. By Lemma \ref{Lemma:P}, it gives us a $\mathscr P$-triple $(\Hom(M,M),f,g)$. Then by Proposition \ref{Prop:Universal}, we get a unique DG-algebra map $\phi: A^{ue}\to \Hom(M,M)$ of degree 0 making certain diagram bi-commute. Hence, we can view $M$ as a DG $A^{ue}$-module via $\phi$. Moreover, it is direct to show that $\phi$ respects homomorphisms between any two DG Poisson $A$-modules. Hence, we have a functor $\Phi:  \DGP$-$A\to \DGM$-$ A^{ue}$ induced by $\phi$. Finally, using the canonical triple $(A^{ue},M,H)$, we see that any $M\in \DGM$-$A^{ue}$ can be viewed as a DG Poisson $A$-module by Lemma \ref{Lemma:P} again, which yields an inverse functor of $\Phi$.
\end{proof}

In the following, we will use $\DGA$ to denote the category of all DG-algebras, and use $\DGPA[p]$ to denote the category of all DG Poison algebras whose Poisson brackets are of degree $p$ along with standard morphisms. It is well-known that how to take opposite algebra and tensor product in $\DGA$ with application of Koszul signs. In $\DGPA[p]$, the opposite DG Poisson algebra of $A$ is denoted by $A^{op}$, where $A^{op}=A$ as DG-algebras and the Poisson bracket of $A^{op}$ is given by 
$$\{a,b\}_{A^{op}}:=(-1)^{(|a|+p)(|b|+p)}\{b,a\}_A=-\{a.b\}_A,\ \mbox{for all}\ a,b\in A^{op}.$$ 
Now, let $B$ be another DG Poisson algebra with Poisson bracket $\{-,-\}_B$ of same degree $p$. The tensor product $A\otimes B$ in $\DGA$ is again a graded commutative DG-algebra, where the product of homogenous elements of degree $d_1$ and $d_2$ is twisted by $(-1)^{d_1d_2}$. Moreover, it is straightforward to check that $A\otimes B$ becomes a DG Poisson algebra with Poisson bracket of degree $p$ by setting 
\begin{align}\label{E:PB}
\{a_1\otimes b_1,a_2\otimes b_2\}_{A\otimes B}:=(-1)^{(|a_2|+p)|b_1|}\{a_1,a_2\}_A\otimes b_1b_2+(-1)^{(|b_1|+p)|a_2|}a_1a_2\otimes \{b_1,b_2\}_B
\end{align}
for any $a_i\in A$ and $b_i\in B, i=1,2$.

\begin{exa}
{\rm Let $A$ be a DG Poisson algebra with Poisson bracket of degree $p$, and $(M,*,\{-,-\}_M,d_M)$ a left DG Poisson module over $A$. Then $(M^{op},*_{op},\{-,-\}_{M^{op}},d_{M^{op}})$ is a right DG Poisson module over the opposite DG Poisson algebra $A^{op}$, where $(M^{op},d_{M^{op}})=(M,d_M)$ as differential graded vector spaces, and 
\[
m*_{op} a:=(-1)^{|a||m|}a*m,\ \{m,a\}_{M^{op}}:=(-1)^{(|a|+p)(|m|+p)}\{a,m\}_M,
\]
for any $a\in A,m\in M$. Moreover, let $B$ be another DG Poisson algebra with Poisson bracket of same degree $p$, and $(N,*,\{-,-\}_N,d_N)$ a left DG Poisson module over $B$. Then, the tensor product $M\otimes N$ can be viewed as a left DG Poisson module over the DG Poisson tensor product $A\otimes B$, where we use the obvious DG $A\otimes B$-module structure on $M\otimes N$, and set up
\[
\{a\otimes b,m\otimes n\}_{M\otimes N}:=(-1)^{(|m|+p)|b|}\{a,m\}_M\otimes (b*n)+(-)^{|m|(|b|+p)}(a*m)\otimes \{b,n\}_N,
\]
for any $a\in A,b\in B,m\in M,n\in N$.
}
\end{exa}

\begin{lemma}\label{4.4}
Let $A,B$ be two DG Poisson algebras with Poisson brackets of degree $p$. Suppose there are two $\mathscr P$-triples $(C,f,g)$ and $(D,j,k)$ with respect to $A,B$ correspondingly. Then, there is a tensor $\mathscr P$-triple
\[
T:=(C\otimes D,f\otimes j,f\otimes k+(-1)^{p|B|}g\otimes j)
\]
with respect to $A\otimes B$. We use $(-1)^{p|B|} g\otimes j$ to mean that $((-1)^{p|B|} g\otimes j)(a\otimes b)=(-1)^{p|b|} g(a)\otimes j(b)$ for any $a\in A,b\in B$.
\end{lemma}
\begin{proof}
It is tedious to check that (P1)-(P4) all hold for the triple $T$ defined above. We only check (P1) and (P3) here, and leave the rest to the readers. (P1): first of all, we show that $f\otimes j: A\otimes B\to C\otimes D$ is a graded algebra map of degree 0. For any $a_i\in A, b_i\in B$, $i=1,2$, we have
\begin{align*}
(f\otimes j)[(a_1\otimes b_1)(a_2\otimes b_2)]&=(-1)^{|a_2||b_1|}(f\otimes j)(a_1a_2\otimes b_1b_2)\\
 &= (-1)^{|a_2||b_1|}f(a_1a_2)\otimes j(b_1b_2)\\
 &= (-1)^{|f(a_2)||j(b_1)|}f(a_1)f(a_2)\otimes j(b_1)j(b_2)\\
 &=(f(a_1)\otimes j(b_1))(f(a_2)\otimes j(b_2))\\
 &=(f\otimes j)(a_1\otimes b_1)(f\otimes j)(a_2\otimes b_2).
\end{align*}
Then, we have to show that $f\otimes j$ commutes the differentials. We have, for any $a\in A,b\in B$,
\begin{align*}
(f\otimes j)d_{A\otimes B}(a\otimes b)&=(f\otimes j)(d_A(a)\otimes b+(-1)^{|a|}a\otimes d_B(b))\\
&=fd_A(a)\otimes j(b)+(-1)^{|a|}f(a)\otimes jd_B(b)\\
&=d_Cf(a)\otimes j(b)+(-1)^{|a|}f(a)\otimes d_Dj(b)\\
&=d_{C\otimes D}(f\otimes j)(a\otimes b).
\end{align*}
(P3): for the left side of the identity, by Equation \eqref{E:PB}, we have
\begin{align*}
&(f\otimes j)(\{a_1\otimes b_1,a_2\otimes b_2\}_{A\otimes B})\\
=&(-1)^{(|a_2|+p)|b_1|}f(\{a_1,a_2\}_A)\otimes j(b_1b_2)+(-1)^{|a_2|(|b_1|+p)}f(a_1a_2)\otimes j(\{b_1,b_2\}_B)\\
=&(-1)^{(|a_2|+p)|b_1|}(g(a_1)f(a_2)-(-1)^{(|a_1|+p)|a_2|}f(a_2)g(a_1))\otimes j(b_1)j(b_2)+\\
&(-1)^{|a_2|(|b_1|+p)}f(a_1)f(a_2)\otimes (k(b_1)j(b_2)-(-1)^{(|b_1|+p)|b_2|}j(b_2)k(b_1))\\
=&(-1)^{|a_2||b_1|+p|b_1|}g(a_1)f(a_2)\otimes j(b_1)j(b_2)-(-1)^{|a_1||a_2|+p|b_1|+p|a_2|+|a_2||b_1|}f(a_2)g(a_1)\otimes j(b_1)j(b_2)\\
&+(-1)^{|a_2||b_1|+p|a_2|}f(a_1)f(a_2)\otimes k(b_1)j(b_2)-(-1)^{|b_1||b_2|+p|b_2|+p|a_2|+|a_2||b_1|}f(a_1)f(a_2)\otimes j(b_2)k(b_1).
\end{align*}
For the right side of the identity, let $\chi=f\otimes k+(-1)^{p|B|}g\otimes j$, we have 
\begin{align*}
&\chi(a_1\otimes b_1)(f\otimes j)(a_2\otimes b_2)-(-1)^{(|a_1|+|b_1|+p)(|a_2|+|b_2|)}(f\otimes j)(a_2\otimes b_2)\chi(a_1\otimes b_1)\\
=&(f(a_1)\otimes k(b_1)+(-1)^{p|b_1|}g(a_1)\otimes j(b_1))(f(a_2)\otimes j(b_2))-(-1)^{(|a_1|+|b_1|+p)(|a_2|+|b_2|)}(f(a_2)\otimes j(b_2))\\
&(f(a_1)\otimes k(b_1)+(-1)^{p|b_1|}g(a_1)\otimes j(b_1))\\
=&(-1)^{|f(a_2)||k(b_1)|}f(a_1)f(a_2)\otimes k(b_1)j(b_2)+(-1)^{p|b_1|+|f(a_2)||j(b_1)|}g(a_1)f(a_2)\otimes j(b_1)j(b_2)\\
&-(-1)^{(|a_1|+|b_1|+p)(|a_2|+|b_2|)+|f(a_1)||j(b_2)|}f(a_2)f(a_1)\otimes j(b_2)k(b_1)\\
&-(-1)^{(|a_1|+|b_1|+p)(|a_2|+|b_2|)+p|b_1|+|g(a_1)||j(b_2)|}f(a_2)g(a_1)\otimes j(b_2)j(b_1).
\end{align*}
Note that $f(a_2)f(a_1)=(-1)^{|f(a_1)||f(a_2)|}f(a_1)f(a_2), j(b_2)j(b_1)=(-1)^{|j(b_1)||j(b_2)|}j(b_1)j(b_2)$, and $f,j$ have degree 0, $g,k$ have degree $p$. Replacing these things in the last equality, we will see that (P3) holds.

\end{proof}

\begin{thm}\label{THM:T}
Let $A,B$ be two DG Poisson algebras with Poisson brackets of the same degree $p$. Then we have
\begin{itemize}
\item[(i)] $(A^{op})^{ue}\cong (A^{ue})^{op}$.
\item[(ii)] $(A\otimes B)^{ue}\cong A^{ue}\otimes B^{ue}$.
\item[(iii)] $(A\otimes A^{op})^{ue}\cong A^{ue}\otimes (A^{ue})^{op}$.
\end{itemize}
Moreover, $ue: \DGPA[p]\to \DGA$ is a tensor functor.
\end{thm}
\begin{proof}
We only prove (ii) here. We can get (i) from the same fashion, and (iii) is a corollary of (i) and (ii). First of all, note that we have two canonical $\mathscr P$-triples $(A^{ue},M_A,H_A)$ and $(B^{ue},M_B,H_B)$. By Lemma \ref{4.4}, there is a tensor $\mathscr P$-triple:
\[
T:=(A^{ue}\otimes B^{ue},M_A\otimes M_B,M_A\otimes H_B+(-1)^{p|B|}H_A\otimes M_B)
\]
with respect to $A\otimes B$. It suffices to show that $T$ has the universal property stated in Proposition \ref{Prop:Universal}. Then the uniqueness argument will imply (ii). Now, suppose $(D,f,g)$ is any $\mathscr P$-triple. We will proceed by completing the following diagram:
\[
\xymatrix{
& A\ar[dl]_-{i_A}\ar[dr]^-{f_A,g_A}\ar[rr]^-{M_A,H_A} & & A^e\ar@{-->}[dl]^-{\exists !\phi_{A}}\ar[dr]^-{i_{A^{ue}}}&\\
A\otimes B\ar[rr]^-{f,g}&  &D & & A^{ue}\otimes B^{ue}\ar@{-->}[ll]_{\exists !\phi}\\
& B\ar[ur]_-{f_B,g_B}\ar[ul]^-{i_B}\ar[rr]_-{M_B,H_B}& & B^{ue}\ar@{-->}[ul]_-{\exists !\phi_{B}}\ar[ur]_-{i_{B^{ue}}}&\\
}
\]
By the natural inclusion $i_A: A\to A\otimes B$, it is clear that the triple $(D,f_A,g_A)$ has property $\mathscr P$ with respect to $A$. Then there exists a unique DG-algebra map $\phi_A: A^{ue}\to D$ of degree 0 such that $f_A=\phi_AM_A$ and $g_A=\phi_AH_A$. Similarly, there exists a unique DG-algebra map $\phi_B: B^{ue}\to D$ of degree 0 such that $f_B=\phi_BM_B$ and $g_B=\phi_BH_B$. Once we establish the following lemma, by the universal property of the tensor product, we will have a unique DG-algebra map $\phi: A^{ue}\otimes B^{ue}\to D$ of degree 0 making the above diagram bi-commutative. Finally, the uniqueness of $\phi$ is clear by tracking the diagram.  
\end{proof}

\begin{lemma}
In the above diagram, for any $x\in A^{ue},y\in B^{ue}$, we have $\phi_A(x)\phi_B(y)=(-1)^{|x||y|}\phi_B(y)\phi_A(x)$.
\end{lemma}
\begin{proof}
By abuse of language, we use the same notations $M,H$ to denote the two linear maps from $A$ and $B$ to $A^{ue}$ and $B^{ue}$ correspondingly. Note that it suffices to consider $x=M_a,H_a$ and $y=M_b,H_b$ for any $a\in A,b\in B$. We only check for $x=M_a$ and $y=M_b$ here, and leave the rest to the readers. We have
\begin{align*}
\phi_A(M_a)\phi_B(M_b)=f_A(a)f_B(b)=f(a\otimes 1)f(1\otimes b)=f(a\otimes b)\\=(-1)^{|a||b|}f(1\otimes b)f(a\otimes 1)=(-1)^{|a||b|}\phi_B(M_b)\phi_A(M_a).
\end{align*}
\end{proof}

\section{Further examples of DG Poisson algebras}\label{Section:Example}
In this section, examples of DG Poisson algebras are provided arising from Lie theory, differential geometry, homological algebra and deformation theory. See more examples in \cite{BM,CFL}.
\subsection{Graded symmetric algebras of DG Lie algebras}
Let $(L,d_L,[-,-]_L)$ be a DG Lie algebra. As in Example \ref{EX:DGV}, we use $S(L)$ to denote the graded symmetric algebra over $L$, i.e.,
\[
S(L):=T(L)/(a\otimes b-(-1)^{|a||b|}b\otimes a),
\]
for any $a,b\in L$. The differential $d_L$ of $L$ can be extended to the graded symmetric algebra $S(L)$ such that $S(L)$ becomes a graded commutative DG-algebra. Here, we consider the total grading on $S(L)$ coming from the grading of $L$. Moreover, the Lie bracket on $L$ also can be extended to a Poisson bracket on $S(L)$ by Jacobi identity such that 
\[
\{a,b\}_{S(L)}:=[a,b]_L,
\]
for any $a,b\in L$. Hence, the graded symmetric algebra $S(L)$ over any DG Lie algebra $L$ has a natural DG Poisson algebra structure. 

Next, note that we can view $L$ as a left DG Lie $L$-module via the Lie bracket $[-,-]_L$. The graded semi-product $L\rtimes L$ is constructed in the following way: (1) $L\rtimes L=L\bigoplus L$ as graded vector spaces, (2) Lie bracket $[-,-]_{L\rtimes L}$ is 0 on the Lie subalgebra $L+0$ and equals $[-,-]_L$ from the pair of components $0+L$ and $L+0$ to $L+0$, (3) the differential on $L\rtimes L$ is given by $d_L$ component-wisely. We denote by $\mathcal U(L\rtimes L)$ the graded universal enveloping algebra of $L\rtimes L$, which is a DG-algebra. Therefore, we have the following isomorphism:
\[
S(L)^{ue}\cong \mathcal U(L\rtimes L).
\]
The isomorphism $\Phi:S(L)^{ue}\to \mathcal U(L\rtimes L)$ can be constructed explicitly by Definition \ref{Def:DGPA} such that $\Phi(M_L)=L+0$ and $\Phi(H_L)=0+L$ in $L\rtimes L$. Moreover, $S(L)$ has a differential graded Hopf structure which is compatible with the Poisson bracket. See \cite[Proposition 6.3]{LWZ1} for the trivial case when $L$ is only a Lie algebra without DG structure.

\subsection{Gerstenhaber algebras}
Consider DG Poisson algebras with trivial differential. In general, most of such examples appear to be in the context of Gerstenhaber algebra, which by definition is a graded commutative algebra $G=\bigoplus_{i\in \mathbb Z} G^i$ together with a Lie bracket $[-,-]:G\otimes G\to G$ of degree $-1$ satisfying: (1) (antisymmetry) $[a,b]=-(-1)^{(|a|-1)(|b|-1)}[b,a]$, (2) (Poisson identity) $[a,bc]=[a,b]c+(-1)^{(|a|-1)|b|}b[a,c]$, (3) (Jacobi identity) $[a,[b,c]]=[[a,b],c]+(-1)^{(|a|-1)(|b|-1)}[b,[a,c]]$ for any $a,b,c\in G$. 

For instance, in homological algebra, Hochschild cohomology of any associated algebra together with the Gerstenhaber Lie bracket \cite{Ger}; in differential geometry, the alternating multivector fields on a smooth manifold together with the Schouten-Nijenhuis bracket extending the Lie bracket of vector fields \cite{Nij}, are all Gertenhaber algebras. Moreover, given an arbitrary Lie-Rinehart algebra $(R,L)$ \cite{Ri}, consider the graded exterior $R$-algebra $\Lambda_R^*L$ over $L$. The Lie bracket $[-.-]$ of $L$ induces a Gerstenhaber algebra structure on $\Lambda_R^*L$ explicitly given by, for any $u=\alpha_1\wedge\cdots\wedge \alpha_l\in \Lambda_R^lL$ and $v=\alpha_{l+1}\wedge\cdots\wedge \alpha_n\in \Lambda_R^{n-l}L$, where $\alpha_1,\cdots,\alpha_n\in L$, then 
\[
[u,v]=(-1)^{|u|}\sum\limits_{j\le l<k}(-1)^{(j+k)}[\alpha_j,\alpha_k]\wedge \alpha_1\wedge\cdots\widehat{\alpha_j}\cdots\widehat{\alpha_k}\cdots \wedge \alpha_n. 
\]

Suppose $A$ is a DG Poisson algebra with Poisson bracket of degree $-1$. Then, it is clear that the cohomology ring $\HG A$ is a Gerstenhaber algebra whose Lie bracket is induced by the Poisson bracket on the cohomology. In the other hand, let $G=\bigoplus_{i\in \mathbb Z} G^i$ be a Gerstenhaber algebra with a Lie bracket $[-,-]$ of degree $-1$. Suppose there is some element $\alpha\in G^2$ satisfying $[\alpha,\alpha]=0$. Hence we can define a map $d: G\to G$ of degree $1$ such that $d: =[\alpha,-]$. If $\mbox{char} \field\neq 2$, one sees that $d^2=0$, hence we have a differential $d$ on $G$. Moreover, it is straightforward to check that, for any $a,b\in G$, we have (1) $d(ab)=d(a)b+(-1)^{|a|}ad(b)$, and (2) $d[a,b]=[d(a),b]+(-1)^{|a|-1}[a,d(b)]$. Therefore, $G$ together with $d=[\alpha,-]$ is a DG Poisson algebra.
  
\subsection{Deformations of graded commutative DG-algebras}
Let $A=\bigoplus_{i\ge \mathbb Z} A^i$ be a graded commutative DG-algebra  together with differential $d$ of degree $1$. Consider a graded deformation of $A$ over the ring $\field[[\hbar]]$, which contains a bimodule map
\[
\star: A[[\hbar]]\otimes A[[\hbar]]\to A[[\hbar]]
\]
making $A[[\hbar]]$ into a DG-algebra where $\deg \hbar=0$. We can write the product of two elements $a,b\in A$ as
\[
a\star b=ab+B_1(a,b)\hbar+B_2(a,b)\hbar^2+\cdots+B_i(a,b)\hbar^i+\cdots,
\]
for bilinear maps $B_i:A\otimes A\to A$ of degree $0$. Suppose $A[[\hbar]]$ is not graded commutative. Then there is some smallest integer $m\ge 1$ such that $B_m(a,b)\neq (-1)^{|a||b|}B_m(b,a)$ for some $a,b\in A$. As a consequence, we can define a Poisson bracket associated to that deformation: for $a,b\in A$, we have 
\[
\{a,b\}:=\{\hbar^{-m}(a\star b-(-1)^{|a||b|}b\star a)\}_{\hbar=0}=B_{m}(a,b)-(-1)^{|a||b|}B_m(b,a).
\]
It is routine to check that this degree 0 bilinear map $\{-,-\}: A\otimes A\to A$ makes $A$ into a DG Poisson algebra.

\section{Closing discussions}\label{Section:Dis}
We close this paper by proposing some questions to be considered in future projects:
\begin{itemize}
\item[(1)] In Xu's paper \cite{Xup}, a noncommutative Poisson structure on an associative algebra $A$ was defined to be any cohomology class $\pi\in \text{H}^2(A,A)$ in the second Hochschild cohomology group of $A$ satisfying $[\pi,\pi]=0$, where $[-,-]$ is the Gerstenhaber bracket on $\text{H}^\bullet(A,A)$. Can we define noncommutative DG Poisson algebras using the same idea?
\item[(2)] Let $A$ be a DG Poisson algebra. One sees easily that there exists a natural graded Poisson algebra map $f_A:\text{H}(A)^{ue}\to \text{H}(A^{ue})$, where $\text{H}$ means taking the cohomology ring. We ask when is $f_A$ an isomorphism? Moreover, we ask if $f: A\to B$ is a quasi-isomorphism of DG Poisson algebras, does $f$ always induce a quasi-isomorphism between $A^{ue}$ and $B^{ue}$?
\item[(3)] The classical Beilinson-Ginzburg-Soergel Koszul duality states a derived equivalence between a Koszul algebra and its Koszul dual. Later, it is generalized to the DG setting by Lu-Palmieri-Wu-Zhang in \cite{LPWZ:Koszul} using $A_\infty$-algebras . Can we establish Koszul duality for certain DG Poisson algebras in terms of the derived equivalence of their DG Poisson modules?
\end{itemize}

\end{document}